\documentclass[12pt,reqno]{amsart}
\usepackage[margin=1in]{geometry}
\usepackage{amsmath,amssymb,amsthm,graphicx,amsxtra, setspace}
\usepackage[utf8]{inputenc}
\usepackage{mathrsfs}
\usepackage{hyperref}
\usepackage{xcolor}
\usepackage{upgreek}
\usepackage{mathtools}
\allowdisplaybreaks

\newtheorem{theorem}{Theorem}[section]

\newtheorem{rem}{Remark}[section]

\newtheorem{definition}{Definition}[section]

\newtheorem{Ex}{Example}[section]
\newtheorem{Ass}{Assumption}[section]
\usepackage{latexsym}
\usepackage{hyperref}
\usepackage{graphicx}
\usepackage{epstopdf}

\allowdisplaybreaks

\let\originalleft\left
\let\originalright\right
\renewcommand{\left}{\mathopen{}\mathclose\bgroup\originalleft}
\renewcommand{\right}{\aftergroup\egroup\originalright}

\newcommand{\Addresses}{{
		\footnote{

			\noindent	 \textsuperscript{1,2,3,4} Department of Applied Sciences and Engineering, Indian Institute of Technology Roorkee, Roorkee, 247667, India.

			\noindent  \textit{e-mail\textsuperscript{1}:} \texttt{jkumar@as.iitr.ac.in}
			
			\noindent  \textit{e-mail\textsuperscript{2}:} \texttt{sonia.iitd.21@gmail.com}
			
				\noindent  \textit{e-mail\textsuperscript{3}:} \texttt{sumit@as.iitr.ac.in}
				
			\noindent  \textit{e-mail\textsuperscript{4}:} \texttt{jay.dabas@gmail.com.}

			\textit{Key words:} Total controllability, Non-instantaneous impulse, Functional evolution equations, Nonlocal condition, Banach fixed point theorem.
			
			Mathematics Subject Classification (2020): 34K06, 34A12, 37L05, 93B05.

}}}

\begin{document}
	\title[]{ Total Controllability of nonlocal semilinear functional evolution equations with non-instantaneous impulses	\Addresses}
	\author [J. kumar, S. Singh, S. Arora \& J. Dabas]{J. kumar\textsuperscript{1}, S. Singh\textsuperscript{2}, S. Arora\textsuperscript{3} and  J. Dabas\textsuperscript{4*}}
	
\maketitle{}
\begin{abstract}
	In this article, we are discussing a more vital concept of controllability, termed total controllability. We have considered a nonlocal semilinear functional evolution equations with non-instantaneous impulses and finite delay in Hilbert spaces.  A set of sufficient conditions of total controllability is obtained for the evolution system under consideration, by imposing the theory of $C_0$-semigroup and Banach fixed point theorem. We also established the total controllability results for a functional integro-differential equation. Finally, an example is given to demonstrate the feasibility of derived abstract results.
\end{abstract}
\section{Introduction}\label{intro}\setcounter{equation}{0}

Many phenomena and processes that experience abrupt changes in their state at some instants during the evolution process. Such processes are appropriately modelled by impulsive differential equations. Impulsive differential equations have attracted much attention because of their applications in control theory, economics, electrical engineering, biology, (cf.	\cite{azbelev2007introduction, chang2007controllability, FM} etc. The classical instantaneous impulses are not capable to explaining the certain dynamics of evolutionary processes in pharmacotherapy. For examples, the introduction of drugs into the bloodstream, in hydrodynamic equilibrium of a person and the consequent immersion for the body are gradual and continuous processes. Such situations can be suitably described by a new type of impulses, known as non-instantaneous impulses, which start at a fixed point and stay active over a finite time interval. In \cite{EH}, Hernández and O' Regan introduced evolution equations with non-instantaneous impulses. Later Wang et al. in \cite{FM}, modified the impulsive conditions considered in \cite{EH}. Recently, Pierri et al. \cite{PM,Pierri2013}, established the global solutions for a class of impulsive abstract differential equations with non-instantaneous impulses. In \cite{Zhang2017}, Zhang et al., considered an evolution system with non-instantaneous impulses and obtained the existence of extremal mild solutions between upper and lower mild solutions by constructing a new monotone iterative method.

On the other hand, there are several real-world phenomena in which the current state of a system depends on the past states and described by the delay differential equations. For example, the logistic reaction-diffusion model with delay, heat conduction in materials with fading memory, neural networks, inferred grinding models and automatic control systems etc, (cf \cite{FX,LA,JW}). For the basic theory of delay differential equations, one can refer to \cite{azbelev2007introduction,hale2013introduction}.

The notion of controllability plays a significant role in designing and examining the control systems. Controllability refer that the solution of a dynamical control system can steer from an arbitrary initial state to a desired final state by using some suitable control function The controllability of infinite-dimensional control systems through various fixed point theorems is studied extensively, see for instance,\cite{balachandran2000controllability,benchohra2004controllability,chang2007controllability,ji2011controllability,JK,AP}, etc and the references therein. In \cite{balachandran2000controllability}, Balachandran investigated the controllability condition for a neutral functional integro-differential system. Benchohra et al., \cite{benchohra2004controllability}, have utilized Schaefer's fixed point theorem to prove the controllability results of the impulsive functional differential inclusions with nonlocal conditions. Chang \cite{chang2007controllability}, derived sufficient conditions  for an impulsive functional differential system with infinite delay to be controllable. In \cite{ji2011controllability}, authors concerned the controllability of the impulsive functional differential equations with nonlocal conditions, and established sufficient conditions of the controllability by using the measure of noncompactness and Mönch fixed-point theorem. In \cite{PJY}, Balachandran et al., described sufficient conditions of the controllability for the impulsive neutral integro-differential systems having infinite delay. Lijuan et al. \cite{SLS}, proved complete controllability for a class of impulsive stochastic integro-differential system using Schaefer's fixed point theorem. Most of the articles in the literature on the controllability , explore the dynamical systems using traditional initial conditions, whereas the nonlocal conditions are more effective to describing countless practical systems. The concept of nonlocal conditions was first introduced by Byszewski in \cite{Byszewski91}, and established the existence and uniqueness of solution to a semilinear evolution equation. Several works have been reported for various systems with nonlocal conditions, see for instance, (\cite{benchohra2000existence, Byszewski91, LL, ntouyas1997global}, etc and the references therein.

Most of the paper discussed above have considered instantaneous impulses. However, there are very few works have been reported on the impulsive control system with non-instantaneous impulses. In \cite{ KA, MM, Muslim2018} etc, considered nonlinear control system with non-instantaneous impulses to establish existence and stability of solutions. Recently, Ahmed et al. \cite{HMA}, considered a non-instantaneous impulsive Hilfer fractional neutral stochastic integro-differential equation with fractional Brownian motion and nonlocal condition. They established the approximate controllability of the considered system by invoking the Sadovskii's fixed point theorem. In \cite{D Chalishjar}, sufficient conditions have been established for the total controllability of the second order semi-linear differential equation with infinite delay and non-instantaneous impulses. Muslim et al. \cite{MV, VM}, established the total controllability and observability results for a dynamic system with non-instantaneous impulses on time scales in finite dimensional space and some necessary and sufficient conditions of the total controllability results for a class of time-varying switched dynamical systems with impulses on time scales. The tools used for that study included Gramian matrix. Wang et al. \cite{WJ}, proved the controllability of fractional non-instantaneous nonlinear impulsive differential inclusions by embedding control only in the last time subinterval. Although, in this paper, we have applied the control for each sub interval of time, which gives rise to the essence of total controllability. Moreover, to the best of our knowledge, no work has been reported so far concerning the total controllability of functional semilinear differential systems with non-instantaneous impulses and nonlocal initial condition. The present study is motivated by this fact. This article is devoted to the total controllability for the nonlocal semilinear functional differential systems with non-instantaneous impulses of the form:
\begin{equation}\label{eqn:e1}
 \left\{
 \begin{aligned}
  \varkappa'(\theta)&=A\varkappa(\theta)+Bu(\theta)+\eta(\theta,\varkappa_\theta),\; \theta\in (\lambda_{j},\theta_{j+1}),\;  j=0,\dots,n, \\
 \varkappa(\theta)&=\nu_{j}(\theta,\varkappa(\theta^{-}_{j})),\;  \theta\in(\theta_{j},\lambda_{j}],\;  j=1,\dots,n, \\
 \varkappa(\theta)&=\phi(\theta),\; \theta \in [-\beta,0],\;\beta>0,\\
 \varkappa(0)&=\phi(0)+ \nu(\varkappa),
 \end{aligned}
 \right.
 \end{equation}
 where the state variable $ \varkappa(\cdot)$ lies in the Hilbert space $\mathbb{X}$ and the linear operator $A$ generates a $C_0$-semigroup $ \mathrm{T}(\theta) $ on $ \mathbb{X}$. The operator $ B : \mathbb{U} \to \mathbb{X} $ is a bounded linear operator and the control function $ u(\cdot) \in L^{2}(J;\mathbb{U}) $ where $ \mathbb{U} $ is also a Hilbert space. The function $\phi $ of the initial condition lies in the space $D$. Also, the nonlinear function $ \eta: J_{1}=\bigcup^{n}_{j=0}[\lambda_{j},\theta_{j+1}]\times D \rightarrow \mathbb{X} $, where 
\begin{align}\label{12}
 D:&=\{\phi: [-\beta,0]\rightarrow \mathbb{X}: \phi  \ \mbox{is piecewise continuous having jump discontinuity at finite number}\ \nonumber\\ &\qquad\mbox{of fixed points}\  \{\beta_0,\beta_1,........\beta_{l+1}\}\subset[-\beta,0], \ \mbox{such that}\ -\beta=\beta_0\leq\beta_1<\beta_2......<\beta_l\nonumber\\&\qquad<\beta_{l+1}=0 \},
\end{align} 

endowed with the norm $\|\phi\|_{D}=\frac{1}{\beta}\int_{-\beta}^{0}\|\phi(\kappa)\|_{\mathbb{X}}d\kappa $.
The non-instantaneous impulsive functions $ \nu_{j}(\theta,\varkappa(\theta^{-}_{j})),\ \theta\in(\theta_{j},\lambda_{j}], \ \mbox{for}\ j=1,\dots,n,$ will specify later and the fixed points $\lambda_i$ and $\theta_j$ satisfy $ 0 = \lambda_{0}=\theta_{0}<\theta_{1}<\lambda_1<\theta_2<,\dots,\theta_{n}<\lambda_n<\theta_{n+1}=b$. The right and left limit of $\varkappa(\cdot)$ at the point $\theta = \theta_j$ for $ j=1,\dots,n$,  exist and denoted by $\varkappa(\theta^{-}_{j})$ and $\varkappa(\theta^{+}_{j})$ respectively, also $\varkappa(\theta^{-}_{j})=\varkappa (\theta_{j})$. For every $\theta \in J=[0,b],$ the function $\varkappa_{\theta}(\kappa) = \varkappa(\theta+\kappa), \;-\beta\leq \kappa \leq 0$ such that $\varkappa_\theta\in D$. The function $\nu:PC(J_{\beta}=[-\beta,b];\mathbb{X})\rightarrow \mathbb{X}$.

The rest of the article is organized as follows: In section \ref{pre}, we provide the notations, definitions, assumptions and results which are required to prove the desired results. In section \ref{Result}, we establish sufficient conditions of the total controllability for the system \eqref{eqn:e1}. The total controllability results are also achieved for a system governed by an integro-differential equation  in Section \ref{sec4}. An illustration to demonstrate the implementation of the findings is also discussed in the last section.

\section{Preliminaries}\label{pre}\setcounter{equation}{0}
In the present section, we provide some fundamental definitions and required assumptions, which is useful to establish sufficient conditions for the total controllability of the system \eqref{eqn:e1}.
The norms in the state space $\mathbb{X}$ and control space $\mathbb{U}$ are denoted by $\|\cdot\|_{\mathbb{X}}$ and $\|\cdot\|_{\mathbb{U}}$ respectively. The space of all bounded linear operators from  $\mathbb{U}$ to  $\mathbb{X}$ is denoted by $\mathcal{L}(\mathbb{U};\mathbb{X})$ equipped with the norm $\|\cdot\|_{\mathcal{L}(\mathbb{U};\mathbb{X})}$. The notation $\mathcal{L}(\mathbb{X})$, represents the space of all bounded linear operators on $\mathbb{X}$ equipped with the norm $\|\cdot\|_{\mathcal{L}(\mathbb{X})}$.

First, we define the following operator \cite{curtain2012introduction}
\begin{equation}\label{resolvent}
\begin{aligned}
\Gamma^{\theta_{j+1}}_{\lambda_{j}}&:=\int_{\lambda_{j}}^{\theta_{j+1}} \mathrm{T}(\theta_{j+1}-\tau)BB^{*} \mathrm{T}(\theta_{j+1}-\tau)^{*}d\tau,\ j=0,\dots,n,
\end{aligned}
\end{equation}
where $ B^{*} $ and $  \mathrm{T}^{*}$ denote the adjoint operators of $ B $ and $  \mathrm{T}$ respectively. 
 
Let us define a set
\begin{align*}
PC(J; \mathbb{X}):&=\{\varkappa:J\rightarrow \mathbb{X} : \varkappa(\cdot) \ \mbox{is piecewise continuous with jump discontinuity at}\nonumber\\&\qquad\mbox{finite number of fixed points}\ \{\theta_0, \theta_1,......\theta_{j+1}\}\subset J\}, \mbox{satisfiying}\nonumber\\&\qquad 0=\theta_0\leq \theta_1,......\theta_j<\theta_{j+1}=b, \ \mbox{with}\ \varkappa(\theta^{-}_{j})= \varkappa(\theta_{j})\ \mbox{for}\ j=1,\dots,n\},
\end{align*}
and the norm on $PC(J;\mathbb{X})$ is defined by $$\|\psi\|_\mathrm{ PC}:=\sup\{\|\psi(\zeta)\|_{\mathbb{X}}, \ 0\leq \zeta \leq b\}.$$ 
The set $ PC(J; \mathbb{X})$ form a Banach space under the norm $\|\cdot\|_\mathrm{ PC} $.

Moreover, we also define a set  $$PC(J_{\beta}; \mathbb{X}):=\{\varkappa:J_{\beta}\rightarrow \mathbb{X}: x|_{\theta\in[-\beta,0)}\in D \ \mbox{and}\ x|_{\theta\in J}\in PC(J;\mathbb{X})\},$$ equipped with the norm $$\|x\|_\mathrm{PCD}=\frac{1}{\beta}\int_{-\beta}^{0}\|\phi(\kappa)\|_{\mathbb{X}}d\kappa+\sup\{\|\psi(\zeta)\|_{\mathbb{X}}, \ 0\leq \zeta \leq b\}.$$ 
It is easily verify that the set  $PC(J_{\beta}; \mathbb{X})$ form a Banach space under the norm $\|\cdot\|_\mathrm{PCD} $.

In order to examine the main results for the system (\ref{eqn:e1}), we impose the following assumptions:
\begin{Ass}\label{as2.1}
\begin{enumerate}
\item [(H1)] There exist constants $ K\geq 1$ and $M$ such that $\| \mathrm{T}(\theta)\|_{\mathcal{L}(\mathbb{X})}\leq K,$ and $\|\mathrm{B}\|_{\mathcal{L(\mathbb{U};\mathbb{X})}}=M$.
\item [(H2)] Let $ \eta : J_{1} \times D \rightarrow \mathbb{X},\ J_{1}=\bigcup^{n}_{j=0}[\lambda_{j},\theta_{j+1}] $ be a continuous function and  for some positive constant $K_{1}$, the following condition holds
$$\|\eta(\theta,\varkappa)-\eta(\theta,y)\|_{\mathbb{X}}\leq K_{1}\|\varkappa-y\|_{D},\ \forall \ \varkappa, y\in {D}, \ \theta\in{J_{1}}.$$
Moreover, there exists a positive constant $N$ such that
$$\|\eta(\theta,\varkappa)\|_{\mathbb{X}}\leq N,\ \forall\; \theta\in J_{1},\ \varkappa\in {D}.$$
\item [(H3)] The non-instantaneous impulses $ \nu_{j}\in \mathrm{C}(I_{j}\times\mathbb{X};\mathbb{X})$, and there exist positive constants $ L_{\nu_{j}},\ \mbox{for}\ j=1,\dots,n$ as
$$ \| \nu_{j}(\theta,\varkappa)-\nu_{j}(\theta,y)\|_{\mathbb{X}} \leq L_{\nu_{j}}\|\varkappa-y\|_{\mathbb{X}},\ \forall \ \theta\in I_{j}=[\theta_{j},\lambda_{j}],\ \varkappa,y\in\mathbb{X}.$$
Moreover, there exist positive constants $C_{\nu_{j}},\ \mbox{for}\ j=1,\dots,n,$ such that
$$\|\nu_{j}(\theta,\varkappa)\|_{\mathbb{X}}\leq C_{\nu_{j}},\ \forall \ \theta\in I_{j}=[\theta_{j},\lambda_{j}],\ \varkappa\in \mathbb{X}. $$
\item [(H4)] The operators $\Gamma^{\theta_{j+1}}_{\lambda_{j}},\ j=0,1,\dots, n,$ defined in \eqref{resolvent}, are invertible and there exist positive constants $\delta^{j}$, such that $\left\|(\Gamma^{\theta_{j+1}}_{\lambda_{j}})^{-1}\right\|_{\mathcal{L}(\mathbb{X})} \le\frac{1}{\delta^{j}},\ j=0,1,\dots, n.$
\item [(H5)] The function $ \nu: PC(J_{\beta};\mathbb{X})\rightarrow \mathbb{X}$ are continuous and there exists a positive constant $C_{\nu}$, such that
$$\|\nu(\varkappa)-\nu(y)\|_{\mathbb{X}}\leq C_{\nu}\|\varkappa-y\|_{\mathrm{PCD}}, \ \text{ for all }\varkappa,y\in PC(J_{\beta};\mathbb{X}).$$
$$ \|\nu(\varkappa)\|\le \bar{M},\ \ \text{ for all }\varkappa\in PC(J_{\beta};\mathbb{X}).$$
\end{enumerate}

\end{Ass}
In view of the articles \cite{AS,FM}, we give the following definition:
\begin{definition}\label{eqn2.3}
A function $ \varkappa:[-\beta,b]\rightarrow\mathbb{X}$ is called a mild solution of the system \eqref{eqn:e1}, if it satisfies the following relations:

\begin{equation}\label{def}
\varkappa(\theta)=\left\{
\begin{aligned}
&\phi(\theta),\ \theta\in[-\beta,0], \ \beta\textgreater 0,\\
&{ \mathrm{T}(\theta)}[\phi(0)+\nu(\varkappa)] + \int_{0}^{\theta} \mathrm{T}(\theta-\tau)\left[Bu(\tau)+\eta(\tau,x_\tau)\right]d\tau, \ \theta\in (0,\theta_{1}],\\
&\nu_{j}(\theta,\varkappa(\theta^{-}_{j})), \ \ \theta\in(\theta_{j},\lambda_{j}],\ j=1,\ldots,n, \\
& \mathrm{T}(\theta-\lambda_{j})(\nu_{j}(\lambda_{j},\varkappa(\theta^{-}_{j}))) + \int_{\lambda_{j}}^{\theta} \mathrm{T}(\theta-\tau)\left[Bu(\tau)+\eta(\tau,\varkappa_\tau)\right]d\tau,\\&\qquad\qquad\quad\ \theta\in (\lambda_{j},\theta_{j+1}],\ j=1,\ldots,n.
\end{aligned}
\right.
\end{equation}
\end{definition}
\begin{definition}
The impulsive system \eqref{eqn:e1} is said to be exactly controllable on $J$, if for the initial state $\varkappa(0)\in D$ and arbitrary final state $\varkappa^{b}\in \mathbb{X}$ there exists a control $u\in L^2(J, \mathbb{U})$, such that the mild solution \eqref{def} satisfies $\varkappa(b)=\varkappa^{b}$.
\end{definition}
Next, to generalize the above definition, we need the following setting.
\begin{definition}
The impulsive system (\ref{eqn:e1}) is said to be totally controllable on $J$, if for the initial
	state $\varkappa(0)\in D $ and arbitrary final state $ \zeta_{j}\in \mathbb{X}$ of each sub-interval $[\lambda_{j}, \theta_{j+1}]$ for $j = 0,\ldots,n$, there exists a
	control $u\in L^2(J,\mathbb{U}),$ such that the mild solution (\ref{def}) satisfies $\varkappa(\theta_{j+1}) = \zeta_{j}$ for $j = 0,\ldots,n$.
\end{definition}
\begin{rem}
If a system is totally controllable on $J$, then it is exactly controllable on $J$. However the converse is not true.
\end{rem}
\section{Total Controllability}\label{Result}
The total controllability of the system \eqref{eqn:e1}, is established in this section. To prove the total controllability of the system, first we obtain the bounds of the  feedback control function. In main theorem \ref{jt3}, the existence of mild solution and total controllability of the system is proved by invoking the Banach fixed point theorem.
	\begin{theorem}\label{jt3}
		Under the Assumptions (H1)-(H5), the control system \eqref{eqn:e1} is totally controllable on $J$, provided that
	\begin{align}\label{LF}
	&\max\bigg\{\max_{1\leq j \leq n}\bigg(\bigg(1+\frac{M^{2}K^{2}b}{\delta^{j}}\bigg)\bigg(KK_{1}\gamma b+ KL_{\nu_{j}}\bigg)\bigg),\nonumber\\&\qquad\qquad\bigg( \bigg(1+\frac{M^{2}K^{2}b}{\delta^{0}}\bigg)\bigg(K_{1}K\gamma b+KC_{\nu}\bigg)\bigg), \max_{1\leq j\leq n}L_{\nu_{j}}\bigg\}<1, \mbox{where}\ \gamma=\frac{b}{\beta}.
	\end{align}
		\end{theorem}
	\begin{proof}
		Let $\alpha_{1}$ defined by
		\begin{align*}\label{eq:control}\alpha_{1}&=\max\bigg\{\bigg(MKQb +KNb+K\left[\|\phi(0)\|_{\mathbb{X}}+\bar{M}\right]\bigg),\max_{1\leq j\leq n}\bigg(MKQb +KNb+KC_{\nu_{j}}\bigg),\nonumber\\&\qquad\qquad \max_{1\leq j\leq n}C_{\nu_{j}}\bigg\}.
		\end{align*}
		Let $\mathcal{Z}:=\{\varkappa\in PC(J;\mathbb{X}):\varkappa(0)=\phi(0)+\nu(\varkappa)\}$ be the space endowed with the norm $ \|\cdot\|_{ PC}$. We define a set
		\begin{align*}
		\mathbb{W}:=\{\varkappa\in \mathcal{Z}: \|\cdot\|_{ PC}\leq \alpha_{1}\},
		\end{align*}
		where $\alpha_{1}$ is a positive constant. Let us define an operator $\xi:\mathcal{Z}\rightarrow\mathcal{Z}$ as
			\begin{align*}
		(\xi\varkappa)(\theta)=z(\theta),
		\end{align*}
		\begin{equation*}
		z(\theta)=
		\left\{
		\begin{aligned}
		\begin{array}{ll}
		 \mathrm{T}(\theta)[\phi(0)+(\nu(\tilde{\varkappa}))]& \\ \quad + \int_{0}^{\theta} \mathrm{T}(\theta-\tau)[Bu(\tau)+\eta(\tau,\tilde{\varkappa}_{\tau})]d\tau,
		\ \theta\in (0,\theta_{1}],\\
		\nu_{j}(\theta,\tilde{\varkappa}(\theta^{-}_{j})),\ \theta\in (\theta_{j},\lambda_{j}],\ j=1,\dots,n,\\
		 \mathrm{T}(\theta-\lambda_{j})\nu_{j}(\lambda_{j},\tilde{\varkappa}(\theta_{j}^{-}))& \\ \quad + \int_{\lambda_{j}}^{\theta} \mathrm{T}(\theta-\tau)[Bu(\tau)+\eta(\tau,\tilde{\varkappa}_{\tau})]d\tau,
		\ \theta\in (\lambda_{j},\theta_{j+1}],\ j=1,\dots,n.
		\end{array}
		\end{aligned}
		\right.
		\end{equation*}
Now, we prove the existence of a fixed point of the operator $\xi$ in the following steps.

{\bf Step 1:} First, let us  define the feedback control as
\begin{eqnarray}\label{eq}
	u(\theta) &=\sum_{j=0}^{n}u_{j}(\theta) \upchi_{ (\lambda_{j},\theta_{j+1}]}(\theta) ,\ \theta\in J,
\end{eqnarray}
where
\begin{eqnarray}
	u_{j}(\theta) &= B^{*} \mathrm{T}(\theta_{j+1}-\theta)^{*} (\Gamma^{\theta_{j+1}}_{\lambda_{j}})^{-1}p_{j}(\varkappa(\cdot)),\  \theta\in (\lambda_{j},\theta_{j+1}],\ j=0,\ldots,n,\nonumber
\end{eqnarray}
with
\begin{align*}
p_{0}(\varkappa(\cdot))&=\zeta_{0}- \mathrm{T}({\theta_{1}})(\phi(0)+(\nu(\tilde{\varkappa }))) -\int_{0}^{\theta_{1}} \mathrm{T}(\theta_{1}-\tau)\eta(\tau,\tilde{\varkappa}_\tau)d\tau,\nonumber\\
p_{j}(\varkappa(\cdot))&=\zeta_{j}- \mathrm{T}({\theta_{j+1}-\lambda_{j}})(\nu_{j}(\lambda_{j},\tilde{\varkappa}(\theta^{-}_{j}))) -\int_{\lambda_{j}}^{\theta_{j+1}} \mathrm{T}(\theta_{j+1}-\tau)\eta(\tau,\tilde{\varkappa}_\tau)d\tau\nonumber\\&\qquad\qquad\qquad\qquad\qquad\qquad\qquad\qquad\qquad \mbox{for}\ j=1,\ldots,n,
\end{align*}
	where, $\tilde{\varkappa}:J_{\beta}\rightarrow\mathbb{X}$ such that $\tilde{\varkappa }(\theta)=\phi(\theta), \theta \in [-\beta,0)$ and $\tilde{\varkappa }(\theta)=\varkappa(\theta), \theta \in J$, and $\zeta_{j}\in\mathbb{X}$ for $j=0,\ldots,n$ are arbitrary. Using \ref{def}, we compute
\begin{align}
\varkappa(\theta_{1})&= \mathrm{T}({\theta_{1}})[\phi(0)+(\nu(\tilde{\varkappa }))]+\int_{0}^{\theta_{1}} \mathrm{T}(\theta_{1}-\tau)[Bu_{0}(\tau)+\eta(\tau,\tilde{\varkappa}_\tau)]d\tau\nonumber\\	
&= \mathrm{T}({\theta_{1}})[\phi(0)+(\nu(\tilde{\varkappa }))] +\int_{0}^{\theta_{1}} \mathrm{T}(\theta_{1}-\tau)\eta(\tau,\tilde{\varkappa}_\tau)d\tau \nonumber\\&\quad+ (\Gamma^{\theta_{1}}_{0})(\Gamma^{\theta_{1}}_{0})^{-1}\bigg[\zeta_{0}- \mathrm{T}({\theta_{1}})[\phi(0)+(\nu(\tilde{\varkappa }))] \nonumber\\&\qquad-\int_{0}^{\theta_{1}} \mathrm{T}(\theta_{1}-\tau)\eta(\tau,\tilde{\varkappa}_\tau)d\tau\bigg]\nonumber\\
&=\zeta_{0}. \nonumber
\end{align}
Moreover, we estimate
\begin{align}
	\varkappa(\theta_{j+1})&= \mathrm{T}({\theta_{j+1}-\lambda_{j}})(\nu_{j}(\lambda_{j},\tilde{\varkappa}(\theta^{-}_{j}))) +\int_{\lambda_{j}}^{\theta_{j+1}} \mathrm{T}(\theta_{j+1}-\tau)[Bu(\tau)+\eta(\tau,\tilde{\varkappa}_\tau)]d\tau\nonumber\\	
	&= \mathrm{T}({\theta_{j+1}-\lambda_{j}})(\nu_{j}(\lambda_{j},\tilde{\varkappa}(\theta^{-}_{j}))) +\int_{\lambda_{j}}^{\theta_{j+1}} \mathrm{T}(\theta_{j+1}-\tau)\eta(\tau,\tilde{\varkappa}_\tau)d\tau \nonumber\\&\quad+ (\Gamma^{\theta_{j+1}}_{\lambda_{j}})(\Gamma^{\theta_{j+1}}_{\lambda_{j}})^{-1}\bigg[\zeta_{j}- \mathrm{T}({\theta_{j+1}-\lambda_{j}})(\nu_{j}(\lambda_{j},\tilde{\varkappa}(\theta^{-}_{j}))) \nonumber\\&\qquad-\int_{\lambda_{j}}^{\theta_{j+1}} \mathrm{T}(\theta_{j+1}-\tau)\eta(\tau,\tilde{\varkappa}_\tau)d\tau\bigg]\nonumber\\
	&=\zeta_{j},\ \mbox{for}\ j=1,\ldots,n.\nonumber
\end{align}
	Hence, the control function \eqref{eq} is suitable for the system \eqref{eqn:e1}. For $\theta\in(0,\theta_1]$, we evaluate 
\begin{align}\label{I}
\|u(\theta)\|_{\mathbb{U}}&\leq \|B^{*}\|_{\mathcal{L}(\mathbb{X};\mathbb{U})}\| \mathrm{T}(\theta_{	 1}-\theta)^{*}\|_{\mathcal{L}(\mathbb{X})}\|(\Gamma^{\theta_{1}}_{0})^{-1}\|_{\mathcal{L}(\mathbb{X})}\Bigg[\|\zeta_{0}\|_{\mathbb{X}}+K\|\phi(0)+(\nu(\tilde{\varkappa }))\|_{\mathbb{X}}\nonumber\\&\quad+\left\|\int_{0}^{\theta_{1}} \mathrm{T}(\theta_{1}-\tau)\eta(\tau,\tilde{\varkappa}_\tau)d\tau\right\|_{\mathbb{X}}\Bigg]\nonumber\\
&\leq \frac{MK}{\delta^{0}}\left[ \|\zeta_{0}\|_{\mathbb{X}}+K\left[\|\phi(0)\|_{\mathbb{X}}+\bar{M}\right] +KNb\right], \nonumber\\&\leq Q_{0}.
\end{align}
For $\theta\in(\lambda_j,\theta_{j+1}],\ j=1,\dots,n$, we estimate
\begin{align}\label{II}
\|u(\theta)\|_{\mathbb{U}}&\leq \|B^{*}\|_{\mathcal{L}(\mathbb{X};\mathbb{U})}\| \mathrm{T}(\theta_{j+	 1}-\theta)^{*}\|_{\mathcal{L}(\mathbb{X})}\|(\Gamma^{\theta_{j+1}}_{\lambda_{j}})^{-1}\|_{\mathcal{L}(\mathbb{X})}\Bigg[\|\zeta_{j}\|_{\mathbb{X}}+K\|\nu_{j}(\lambda_{j},\tilde{\varkappa}(\theta^{-}_{j}))\|_{\mathbb{X}}\nonumber\\&\quad+\left\|\int_{\lambda_{j}}^{\theta_{j+1}} \mathrm{T}(\theta_{j+1}-\tau)\eta(\tau,\tilde{\varkappa}_\tau)d\tau\right\|_{\mathbb{X}}\Bigg]\nonumber\\
&\leq \frac{MK}{\delta^{j}}\left[ \|\zeta_{j}\|_{\mathbb{X}}+KC_{\nu_{j}} +KNb\right].\nonumber\\
&\leq Q_{j}.
\end{align}
Combining \eqref{I} and \eqref{II}, we obtain 
\begin{align*}
\|u(\theta)\|_\mathbb{U}&\leq Q,\ \mbox{for}\ \theta\in J,  \ \mbox{where} \ Q=\max_{0\leq j\leq n}\{Q_{j}\}.
\end{align*}
{\bf Step 2:} Next, we prove that $\xi$ maps bounded set to bounded set. We need to show that for $\alpha_{1}>0$,  $\xi(\mathbb{W})\subseteq\mathbb{W} $. Let us take $\theta\in(\lambda_{j},\theta_{j+1}],\ j=1,\dots,n$ and $y\in \mathbb{W}$, we estimate
		\begin{align*}
			&\|(\xi y)(\theta)\|_{\mathbb{X}}\nonumber\\&\leq\|\mathrm{T}(\theta-\lambda_{j})\nu_{j}(\lambda_{j},\tilde{y}(\theta_{j}^{-}))\|_{\mathbb{X}}+\int_{\lambda_{j}}^{\theta}\| \mathrm{T}(\theta-\tau)Bu(\tau)\|_{\mathbb{X}}d\tau
		+\int_{\lambda_{j}}^{\theta}\| \mathrm{T}(\theta-\tau)\eta(\tau,\tilde{y}_{\tau})\|_{\mathbb{X}}d\tau\nonumber\\&\leq KC_{\nu_{j}}+ MKQb+KNb.
		\end{align*}
		Thus, we have
		\begin{align}\label{3.7}
		\|(\xi y)\|_{ PC}\leq\bigg ( MKQb+KNb+KC_{\nu_{j}}\bigg).
		\end{align}
Now, for $\theta\in[0,\theta_{1}]$ and $y\in \mathbb{W}$, we compute
\begin{align*}
	\|(\xi y)(\theta)\|_{\mathbb{X}}&\leq \| \mathrm{T}(\theta)[\phi(0)+(\nu(\tilde{y}))]\|_{\mathbb{X}}+\int_{0}^{\theta}\| \mathrm{T}(\theta-\tau)Bu(\tau)\|_{\mathbb{X}}d\tau+\int_{0}^{\theta}\| \mathrm{T}(\theta-\tau)\eta(\tau,\tilde{y}_{\tau})\|_{\mathbb{X}}d\tau\nonumber\\&\leq K\left[\|\phi(0)\|_{\mathbb{X}}+\bar{M}\right]+MKQb+KNb.
\end{align*}
	Hence, we obtain
	\begin{align}\label{3.8}
		\|(\xi y)\|_{ PC}\leq \bigg(MKQb +KNb+K\left[\|\phi(0)\|_{\mathbb{X}}+\bar{M}\right]\bigg).
	\end{align}
		Similarly for $\theta\in(\theta_{j},\lambda_{j}],\ j=1,\dots,n$ and $y\in \mathbb{W}$, we obtain
		\begin{align}\label{3.9}
			\|(\xi y)\|_{ PC}\leq C _{\nu_{j}}.
		\end{align}
		Summarizing inequalities \eqref{3.7}, \eqref{3.8}, \eqref{3.9}, we have
		\begin{align*}
			\|(\xi y)\|_{ PC}\leq \alpha_{1}.
		\end{align*}
		
		{\bf Step 3:}		Next, we will prove that the map $\xi$ is a contraction map. For this, let us take any  $\varkappa,y\in \mathbb{W}$, $\theta\in(\lambda_{j},\theta_{j+1}],\ j=1,\dots,n$, we compute
		\begin{align}\label{eqn:3.9}
		&\|(\xi \varkappa)(\theta)-(\xi y)(\theta)\|_{\mathbb{X}}\nonumber\\&\leq\| \mathrm{T}(\theta-\lambda_{j})[\nu_{j}(\lambda_{j},\tilde{\varkappa}(\theta_{j}^{-}))- \nu_{j}(\lambda_{j},\tilde{y}(\theta_{j}^{-})]\|_{\mathbb{X}}+ \int_{\lambda_{j}}^{\theta}\| \mathrm{T}(\theta-\tau)\|_{\mathcal{L}(\mathbb{X})}\|\eta(\tau,\tilde{\varkappa }_{\tau})-\eta(\tau,\tilde{y}_{\tau})\|_\mathbb{X}d\tau\nonumber\\&+\int_{\lambda_{j}}^{\theta}\| \mathrm{T}(\theta-\tau)\|_{\mathcal{L}(\mathbb{X})}\|B\|_{\mathcal{L}(\mathbb{U};\mathbb{X})}\|B^{*}\|_{\mathcal{L}(\mathbb{X};\mathbb{U})}\| \mathrm{T}(\theta_{j+1}-\tau)^{*}\|_{\mathcal{L}(\mathbb{X})}\|(\Gamma_{\lambda_{j}}^{\theta_{j+1}})^{-1}\|\|p_{j}(\varkappa)-p_{j}(y)\|_{\mathbb{X}}d\tau\nonumber\\&\leq \| \mathrm{T}(\theta-\lambda_{j})\|_{\mathcal{L}(\mathbb{X})}\|\nu_{j}(\lambda_{j},\tilde{\varkappa}(\theta_{j}^{-}))-\nu_{j}(\lambda_{j},\tilde{y}(\theta_{j}^{-})\|_{\mathbb{X}}+\int_{\lambda_{j}}^{\theta}KK_{1}\|\tilde{\varkappa }_{\tau}-\tilde{y}_{\tau}\|_{D}d\tau\nonumber\\&\qquad+ \int_{\lambda_{j}}^{\theta}\| \mathrm{T}(\theta-\tau)\|_{\mathcal{L}(\mathbb{X})}\| \mathrm{T}^{*}(\theta_{j+1}-\tau)\|_{\mathcal{L}(\mathbb{X})}
		\frac{M^{2}}{\delta^{j}}\|p_{j}(\varkappa)-p_{j}(y)\|_{\mathbb{X}}d\tau.
	\end{align}
	For $\tau \in J$, let us compute
			\begin{align}\label{eqn:3.8}
		 \|\tilde{\varkappa}_{\tau}-\tilde{y}_{\tau}\|_{D}&=\frac{1}{\beta}\int_{-\beta}^{0}\|\tilde{\varkappa}_{\tau}(\kappa)-\tilde{y}_{\tau}(\kappa)\|_{\mathbb{X}}d\kappa=\frac{1}{\beta}\int_{-\beta}^{0}\|\tilde{\varkappa}(\tau+\kappa)-\tilde{y}(\tau+\kappa)\|_{\mathbb{X}}d\kappa\nonumber\\&=\frac{1}{\beta}\int_{\tau-\beta}^{\tau}\|\tilde{\varkappa}(u)-\tilde{y}(u)\|_{\mathbb{X}}du,\ \text{ where }u=\tau+\kappa.\end{align}
	If $\tau\textless \beta$, then the expression  $\eqref{eqn:3.8}$ is express as
	  \begin{align*}
		   \|\tilde{\varkappa}_{\tau}-\tilde{y}_{\tau}\|_{D}&=\frac{1}{\beta}\int_{\tau-\beta}^{0}\|\tilde{\varkappa}(u)-\tilde {y}(u)\|_{\mathbb{X}}du+\frac{1}{\beta}\int_{0}^{\tau}\|\tilde{\varkappa}(u)-\tilde {y}(u)\|_{\mathbb{X}}du \nonumber\\&\leq\frac{1}{\beta}\int_{-\beta}^{0}\|\tilde{\varkappa}(u)-\tilde {y}(u)\|_{\mathbb{X}}du+\frac{1}{\beta}\int_{0}^{b}\|\tilde{\varkappa}(u)-\tilde {y}(u)\|_{\mathbb{X}}du\nonumber\\&\leq\frac{1}{\beta}\int_{0}^{b}\sup\|\varkappa(u)-y(u)\|_{\mathbb{X}}du\nonumber\\&\leq\|\varkappa-y\|_{ PC}\frac{1}{\beta}\int_{0}^{b} du\nonumber\\&\leq\frac{b}{ \beta}\|\varkappa-y\|_{ PC}\leq\gamma\|\varkappa-y\|_{ PC}, \mbox{ where } \gamma=\frac{b}{\beta}.
		\end{align*}
		If $\tau\geq \beta$, then the expression $\eqref{eqn:3.8}$ is express as
\begin{align*}
  \|\tilde{\varkappa}_{\tau}-\tilde{y}_{\tau}\|_{D}&\le\frac{1}{\beta}\int_{0}^{\tau}\|\tilde{\varkappa}(u)-\tilde {y}(u)\|_{\mathbb{X}}du \nonumber\\&\leq\frac{1}{\beta}\int_{0}^{b}\|\tilde{\varkappa}(u)-\tilde {y}(u)\|_{\mathbb{X}}du\nonumber\\&\leq\frac{1}{\beta}\int_{0}^{b}\sup\|\varkappa(u)-y(u)\|_{\mathbb{X}}du\nonumber\\&\leq\|\varkappa-y\|_{ PC}\frac{1}{\beta}\int_{0}^{b} du\nonumber\\&\leq\frac{b}{\beta}\|\varkappa-y\|_{ PC}\leq\gamma\|\varkappa-y\|_{ PC}.
\end{align*}
Thus, we have
\begin{align}\label{es}
	\|\tilde{\varkappa}_{\tau}-\tilde{y}_{\tau}\|_{D}\le&\gamma\|\varkappa-y\|_{ PC}, \tau\in J.
\end{align}
  Using  Assumptions \ref{as2.1}, we estimate
		\begin{align}\label{eqn:3.10}
		\|p_{j}(\varkappa)-p_{j}(y)\|_{\mathbb{X}}\le & KL_{\nu_{j}}\|\varkappa-y\|_{\mathbb{X}}+KK_{1}b\|\tilde{\varkappa}_{\tau}-\tilde{y}_{\tau}\|_{D}\nonumber\\
		\le  &KL_{\nu_{j}}\|\varkappa-y\|_{\mathbb{X}}+KK_{1}\gamma b\|\varkappa-y\|_{\mathbb{X}},\ \mbox{for}\ j=1,\ldots,n.
		\end{align}
		Combine \eqref{eqn:3.9}, \eqref{es} and \eqref{eqn:3.10}, we have
		\begin{align}
		&\|(\xi \varkappa)(\theta)-(\xi y)(\theta)\|_{\mathbb{X}}\nonumber\\&\leq KL_{\nu_{j}}\|\varkappa-y\|_{\mathbb{X}}+KK_{1}\gamma b\|\varkappa-y\|_{ PC} +\frac{M^{2}K^{2}b}{\delta^{j}}\bigg[KL_{\nu_{j}}\|\varkappa-y\|_{\mathbb{X}}\nonumber\\&\qquad\qquad+KK_{1}\gamma b\|\varkappa-y\|_{ PC}\bigg]\nonumber.
		\end{align}
		By using the above expression, we obtain
		\begin{align}\label{eqn:3.13}
	  &\|(\xi \varkappa)(\theta)-(\xi y)(\theta)\|_{\mathbb{X}}\leq\bigg(\bigg(1+\frac{M^{2}K^{2}b}{\delta^{j}}\bigg)\bigg( KK_{1}b\gamma+ KL_{\nu_{j}}\bigg)\bigg)\|\varkappa-y\|_{ PC}.
		\end{align}
		For $\theta\in[0,\theta_{1}]$ and $\varkappa,y\in \mathbb{W}$, we calculate
		\begin{align}\label{eqn:3.11}
	&\nonumber	\|(\xi \varkappa)(\theta)-(\xi y)(\theta)\|_{\mathbb{X}}\\&\leq \|\mathrm{T}(\theta)[(\nu(\tilde{\varkappa}))-(\nu(\tilde{y}))]\|_{\mathbb{X}}+\int_{0}^{\theta}\| \mathrm{T}(\theta-\tau)\|\|_{\mathcal{L}(\mathbb{X})}\eta(\tau,\tilde{\varkappa }_{\tau})-\eta(\tau,\tilde{y}_{\tau})\|_{\mathbb{X}}d\tau \nonumber\\&\quad+ \int_{0}^{\theta}\| \mathrm{T}(\theta-\tau)\|_{\mathcal{L}(\mathbb{X})}\|B\|_{\mathcal{L}(\mathbb{U;}\mathbb{X})}\|B^{*}\|_{\mathcal{L}(\mathbb{X};\mathbb{U})}\| \mathrm{T}(\theta_{1}-\tau)^{*}\|\|_{\mathcal{L}(\mathbb{X})}(\Gamma_{0}^{\theta_{1}})^{-1}\|\|p_{0}(\varkappa)-p_{0}(y)\|_{\mathbb{X}}d\tau\nonumber\\
			&\leq \|\mathrm{T}(\theta)\|_{\mathcal{L}(\mathbb{X})}\|(\nu(\tilde{\varkappa}))-(\nu(\tilde{y}))\|_{\mathbb{X}}+K_{1}\int_{0}^{\theta}\| \mathrm{T}(\theta-\tau)\|_{\mathcal{L}(\mathbb{X})}\|\tilde{\varkappa}_{\tau}-\tilde{y}_{\tau}\|_{D}d\tau\nonumber\\&\qquad + \frac{M^{2}}{\delta^{0}}\int_{0}^{\theta}\| \mathrm{T}(\theta-\tau)\|_{\mathcal{L}(\mathbb{X})}\| \mathrm{T}(\theta_{1}-\tau)^{*}\|_{\mathcal{L}(\mathbb{X})}\|p_{0}(\varkappa)-p_{0}(y)\|_{\mathbb{X}}d\tau.
		\end{align}
		Further, using Assumptions \ref{as2.1}, we evaluate
		\begin{align}\label{eqn:3.12}
		\|p_{0}(\varkappa)-p_{0}(y)\|_{\mathbb{X}}&\le KC_{\nu}\|\tilde{\varkappa}-\tilde{y}\|_{PCD}+KK_{1}b\|\tilde{\varkappa }_{\tau}-\tilde{y}_{\tau}\|_{D}\nonumber\\&\le KC_{\nu}\|\varkappa-y\|_{ PC}+KK_{1}\gamma b\|\varkappa-y\|_{ PC}.
		\end{align}
		Combine \eqref{es}, \eqref{eqn:3.11} and \eqref{eqn:3.12}, we obtain
		\begin{align}\label{eqn:3.14}
		\|(\xi \varkappa)(\theta)-(\xi y)(\theta)\|_{\mathbb{X}}&\leq\bigg( \bigg(1+\frac{M^{2}K^{2}b}{\delta^{0}}\bigg)\bigg(K_{1}K\gamma b+KC_{\nu}\bigg)\bigg)\|\varkappa-y\|_{ PC}.
		\end{align}
		Similarly for $\theta\in(\theta_{j},\lambda_{j}],\ j=1,\dots,n$ and $\varkappa,y\in \mathbb{W}$, we compute
		\begin{align}
		\|(\xi \varkappa)(\theta)-(\xi y)(\theta)\|_{\mathbb{X}} \leq&\|\nu_{j}(\theta,\tilde{\varkappa }(\theta^{-}_{j}))-\nu_{j}(\theta,\tilde{y}(\theta^{-}_{j}))\|_{\mathbb{X}}\nonumber\\ 	&\leq L_{\nu_{j}}\|\tilde{\varkappa }(\theta^{-}_{j})-\tilde{y}(\theta^{-}_{j})\|_{\mathbb{X}} \nonumber\\
		&\leq L_{\nu_{j}}\|\varkappa-y\|_{\mathbb{X}}.\nonumber
		\end{align}
		Moreover, by the above estimate, we obtain
		\begin{align}\label{eqn:3.15}
		\|(\xi \varkappa)(\theta)-(\xi y)(\theta)\|_{\mathbb{X}}\leq& L_{\nu_{j}}\|\varkappa-y\|_{ PC}.
		\end{align}
		After summarizing the inequalities \eqref{eqn:3.13},\eqref{eqn:3.14} and \eqref{eqn:3.15}, we have
		\begin{align}\label{contr}
		\|(\xi \varkappa)-(\xi y)\|_{ PC}\leq L_{F}\|\varkappa-y\|_{ PC},
		\end{align}
		where
		\begin{align*}
		L_{F}&=\max\bigg\{\max_{1\leq j \leq n}\bigg(\bigg(1+\frac{M^{2}K^{2}b}{\delta^{j}}\bigg)\bigg(KK_{1}\gamma b+ KL_{\nu_{j}}\bigg)\bigg),\bigg( \bigg(1+\frac{M^{2}K^{2}b}{\delta^{0}}\bigg)\\&\qquad\qquad\qquad\bigg(K_{1}K\gamma b+KC_{\nu}\bigg)\bigg), \max_{1\leq j\leq n}L_{\nu_{j}}\bigg\}.
		\end{align*}
	From \eqref{contr} and \eqref{LF}, we conclude	that $\xi$ is a contraction map. This prove the existence of a unique fixed point of $\xi$. Which is a mild solution of (\ref{eqn:e1}). Therefore, the system (\ref{eqn:e1}) is totally controllable.
	\end{proof}
	
\section{Total Controllability Of Integro-Differential Equation}\label{sec4}
In \cite{RK}, Chalishajar et al., studied the trajectory controllability of an abstract nonlinear integro-differential system in the finite and infinite dimensional space. Bahuguna et al. \cite{RD}, established the sufficient conditions for the existence and uniqueness of piecewise continuous mild solutions to fractional integro-differential equations in a Banach space with non-instantaneous impulses using the fixed point theorem. In \cite{ZH}, Zhu et al., considered the initial boundary value problem for a class of nonlinear fractional partial integro-differential equations of mixed type with non-instantaneous impulses in Banach spaces.  Sufficient conditions of existence and uniqueness of mild solutions for the equations are obtained via general Banach contraction mapping principal. 

Here, we consider a functional integro-differential control system with non-instantaneous impulse in Hilbert Space $\mathbb{X}$.
 \begin{equation}\label{eqn:e4}
\left\{
\begin{aligned}
\varkappa'(\theta)&=A\varkappa(\theta)+Bu(\theta)+\int_{0}^{\theta}\kappa(\theta-s)q(s,\theta_{s})ds,\; \theta\in (\lambda_{j},\theta_{j+1}),\;  j=0,\dots,n, \\
\varkappa(\theta)&=\nu_{j}(\theta,\varkappa(\theta^{-}_{j})),\;  \theta\in(\theta_{j},\lambda_{j}],\;  j=1,\dots,n, \\
 \varkappa(\theta)&=\phi(\theta),\; \theta \in [-\beta,0],\;\theta>0,
\end{aligned}
\right.
\end{equation}
where $\varkappa(\cdot)$ is the state variable, $u(\cdot)\in L^2(J;\mathbb{U})$ is the control function, $\mathbb{U}$ is a Hilbert space, $B\in \mathcal{L}(\mathbb{U};\mathbb{X})$. In order to prove the  total controllability of \eqref{eqn:e4}, we need the following additional assumptions:
\begin{Ass}\label{as4.1}
	
	\begin{enumerate}
		\item [(A1)]$ \kappa _{b}=\int_{0}^{\theta}|\kappa(s)|ds $.
		\item [(A2)]  The continuous function $ q: J_{1} \times D \rightarrow \mathbb{X},\ J_{1}=\bigcup^{n}_{j=0}[\lambda_{j},\theta_{j+1}]$,  satisfies
		Lipschitz continuity condition, there exists a positive constant $L_{q}$ such that
		$$\|q(\theta,\varkappa)-q(\theta,y)\|_{\mathbb{X}}\leq L_{q}\|\varkappa-y\|_{D}, \forall \;\varkappa,y\in D,\ \theta\in{J_{1}}.$$
		Moreover, there exists $S>0$ such that
		$$\|q(\theta,\varkappa)\|_{\mathbb{X}}\leq S, \forall\; \theta\in J_{1}, \ \mbox{for}\ \varkappa\in D.$$
		\end{enumerate}
	\end{Ass}
\begin{definition}\label{definition}
	A function $ \varkappa\in PC([-\beta,b];\mathbb{X})$ is called a mild solution of the system \eqref{eqn:e4}, if the following relations are satisfied
	\begin{equation}\label{eqn4.4}
	\varkappa(\theta)=\left\{
	\begin{aligned}
	&\phi(\theta),\ \ \theta \in [-\beta,0],\;\beta>0,\\
	&\nu_{j}(\theta,\varkappa(\theta^{-}_{j})), \ \ \theta\in(\theta_{j},\lambda_{j}],\ j=1,\ldots,n,\\
	& \mathrm{T}(\theta)\phi(0)+ \int_{0}^{\theta} \mathrm{T}(\theta-\tau)\left[Bu(\tau)+\int_{0}^{\tau}\kappa(\tau-\eta)q(\eta,\varkappa_{\eta})d\eta\right]d\tau, \ \theta\in (0,\theta_{1}],\\
	& \mathrm{T}(\theta-\lambda_{j})(\nu_{j}(\lambda_{j},\varkappa(\theta^{-}_{j}))) + \int_{\lambda_{j}}^{\theta}T(\theta-\tau) \Big[Bu(\tau) \\& \quad+ \int_{0}^{\tau}\kappa(\tau-\eta)q(\eta,\varkappa_{\eta})d\eta\Big]d\tau, \theta\in (\lambda_{j},\theta_{j+1}],\ j=1,\ldots,n.
\end{aligned}
	\right.
	\end{equation}
\end{definition}

	\begin{theorem}\label{jt6}
	If the assumptions (H1),(H3)-(H4) and (A1)-(A2) are fulfilled, then the control system \eqref{eqn:e4} is totally controllable provided that
	\begin{align}\label{LF1}
	&\max\bigg\{\max_{1\leq j \leq n}\bigg(\bigg( KL_{\nu_{j}}+KL_{q}\kappa_{b}\gamma b\bigg) \bigg(1+\frac{M^{2}K^{2}b}{\delta^{j}}\bigg)\bigg), \nonumber\\ &\qquad\qquad\bigg(1+\frac{M^{2}K^{2}b}{\delta^{0}}\bigg)KL_{q}\kappa_{b}\gamma b,\ \max_{1\leq j \leq n}L_{\nu_{j}}\bigg\}<1, \mbox{where}\ \gamma=\frac{b}{\beta}.
	\end{align}
\end{theorem}
	\begin{proof}
		Let $\alpha_{2}$ defined by
	\begin{align*}\label{eqn:control}\alpha_{2}&=\max\bigg\{\bigg(K\|\phi(0)\|_{\mathbb{X}} +MKRb +KSb\kappa_{b}\bigg), \max_{1\leq j\leq n}\bigg(KC_{\nu_{j}}\\&\qquad\qquad+  MKRb+KSb\kappa_{b}\bigg), \ \max_{1\leq j\leq n}C_{\nu_{j}}\bigg\}.
	\end{align*}
		Let $\mathcal{A}:=\{\varkappa\in PC(J;\mathbb{X}):\varkappa(0)=\phi(0)\}$ be the space endowed with the norm $ \|\cdot\|_{ PC}$. We define a set
	\begin{align*}
	\mathbb{Y}:=\{\varkappa\in \mathcal{A}: \|\cdot\|_{ PC}\leq \alpha_{2}\},
	\end{align*}
	where $\alpha_{2}$ is a positive constant. Let us define an operator $\delta:\mathcal{A}\rightarrow\mathcal{A}$ as
	\begin{align*}
	(\delta\varkappa)(\theta)=z(\theta),
	\end{align*}
		\begin{equation}
	z(\theta)=
	\left\{
	\begin{aligned}
	\begin{array}{ll}
	\mathrm{T}(\theta)\phi(0)+ \int_{0}^{\theta} \mathrm{T}(\theta-\tau)\left[Bu(\tau)+\int_{0}^{\tau}\kappa(\tau-\eta)q(\eta,\tilde{y}_{\eta})d\eta\right]d\tau, \ \theta\in (0,\theta_{1}],\\
		\nu_{j}(\theta,\tilde{y}(\theta_{j}^{-})),\ \theta\in (\theta_{j},\lambda_{j}],\ j=1,\dots,n,\\
	  \mathrm{T}(\theta-\lambda_{j})(\nu_{j}(\lambda_{j},\tilde{y}(\theta^{-}_{j}))) +\int_{\lambda_{j}}^{\theta} \mathrm{T}(\theta-\tau)\left[Bu(\tau)+\int_{0}^{\tau}\kappa(\tau-\eta)q(\eta,\tilde{y}_{\eta})d\eta\right]d\tau,
\\\qquad\qquad\qquad\qquad	\ \theta\in (\lambda_{j},\theta_{j+1}],\ j=1,\dots,n.
	\end{array}
	\end{aligned}
	\right.\nonumber
	\end{equation}
 Now, we prove the existence of a fixed point of the operator $\delta$ in the following steps.\\
{\bf Step 1}: First, let us  define the feedback control as
\begin{eqnarray}\label{eqn}
u(\theta) &=\sum_{j=0}^{n}u_{j}(\theta) \upchi_{ (\lambda_{j},\theta_{j+1}]}(\theta) ,\ \theta\in J,
\end{eqnarray}
where
\begin{eqnarray}
u_{j}(\theta) &= B^{*} \mathrm{T}(\theta_{j+1}-\theta)^{*} (\Gamma^{\theta_{j+1}}_{\lambda_{j}})^{-1}h_{j}(\varkappa(\cdot)),\  \theta\in (\lambda_{j},\theta_{j+1}],\ j=0,\ldots,n,\nonumber
\end{eqnarray}
with
\begin{align*}
h_{0}(\varkappa(\cdot))&=\zeta_{0}- \mathrm{T}({\theta_{1}})\phi(0) -\int_{0}^{\theta_{1}} \mathrm{T}(\theta_{1}-\tau)\int_{0}^{\tau}\kappa(\tau-\eta)q(\eta,\tilde{\varkappa }_{\eta})d\eta d\tau,\nonumber\\
h_{j}(\varkappa(\cdot))&=\zeta_{j}- \mathrm{T}({\theta_{j+1}-\lambda_{j}})(\nu_{j}(\lambda_{j},\tilde{\varkappa }(\theta^{-}_{j})))-\int_{\lambda_{j}}^{\theta_{j+1}} \mathrm{T}(\theta_{j+1}-\tau)\int_{0}^{\tau}\kappa(\tau-\eta)q(\eta,\tilde{\varkappa }_{\eta})d\eta d\tau\nonumber\\&\qquad\qquad\qquad\qquad\qquad\qquad\qquad\qquad\qquad\qquad\qquad\qquad\qquad \mbox{for}\ j=1,\ldots,n,
\end{align*}
where, $\tilde{\varkappa}:J_{\beta}\rightarrow\mathbb{X}$ such that $\tilde{\varkappa }(\theta)=\phi(\theta), \theta \in [-\beta,0)$ and $\tilde{\varkappa }(\theta)=\varkappa(\theta), \theta \in J$, and $\zeta_{j}\in\mathbb{X}$ for $j=0,\ldots,n$, are arbitrary. Using \eqref{eqn4.4}, we compute
\begin{align}
\nonumber	\varkappa(\theta_{1})&= \mathrm{T}({\theta_{1}})\phi(0)+\int_{0}^{\theta_{1}} \mathrm{T}(\theta_{1}-\tau)\int_{0}^{\tau}\kappa(\tau-\eta)q(\eta,\tilde{\varkappa }_{\eta})d\eta d\tau\nonumber\\\nonumber &\qquad+\int_{0}^{\theta_{1}} \mathrm{T}(\theta_{1}-\tau)Bu(\tau)d\tau\nonumber\\	
&= \mathrm{T}({\theta_{1}})\phi(0)+\int_{0}^{\theta_{1}} \mathrm{T}(\theta_{1}-\tau)\int_{0}^{\tau}\kappa(\tau-\eta)q(\eta,\tilde{\varkappa }_{\eta})d\eta d\tau\nonumber\\&\qquad+ (\Gamma^{\theta_{1}}_{0})(\Gamma^{\theta_{1}}_{0})^{-1}\bigg[\zeta_{0}- \mathrm{T}({\theta_{1}})\phi(0)-\int_{0}^{\theta_{1}} \mathrm{T}(\theta_{1}-\tau)\int_{0}^{\tau}\kappa(\tau-\eta)q(\eta,\tilde{\varkappa }_{\eta})d\eta d\tau\bigg]\nonumber\\
&=\zeta_{0}. \nonumber
\end{align}
Moreover, we estimate
\begin{align}
\nonumber	\varkappa(\theta_{j+1})&= \mathrm{T}({\theta_{j+1}-\lambda_{j}})(\nu_{j}(\lambda_{j},\tilde{\varkappa }(\theta^{-}_{j}))) +\int_{\lambda_{j}}^{\theta_{j+1}} \mathrm{T}(\theta_{j+1}-\tau)\int_{0}^{\tau}\kappa(\tau-\eta)q(\eta,\tilde{\varkappa }_{\eta})d\eta d\tau\nonumber\\\nonumber &\qquad+\int_{\lambda_{j}}^{\theta_{j+1}} \mathrm{T}(\theta_{j+1}-\tau)Bu(\tau)d\tau\nonumber\\	
&= \mathrm{T}({\theta_{j+1}-\lambda_{j}})(\nu_{j}(\lambda_{j},\tilde{\varkappa }(\theta^{-}_{j}))) +\int_{\lambda_{j}}^{\theta_{j+1}} \mathrm{T}(\theta_{j+1}-\tau)\int_{0}^{\tau}\kappa(\tau-\eta)q(\eta,\tilde{\varkappa }_{\eta})d\eta d\tau\nonumber\\&\qquad+ (\Gamma^{\theta_{j+1}}_{\lambda_{j}})(\Gamma^{\theta_{j+1}}_{\lambda_{j}})^{-1}\bigg[\zeta_{j}- \mathrm{T}({\theta_{j+1}-\lambda_{j}})(\nu_{j}(\lambda_{j},\tilde{\varkappa }(\theta^{-}_{j})))\nonumber\\&\qquad-\int_{\lambda_{j}}^{\theta_{j+1}} \mathrm{T}(\theta_{j+1}-\tau)\int_{0}^{\tau}\kappa(\tau-\eta)q(\eta,\tilde{\varkappa }_{\eta})d\eta d\tau\bigg]\nonumber\\
&=\zeta_{j}, \ \mbox{for}\, j=1,\ldots,n. \nonumber
\end{align}
Hence, the control function \eqref{eqn} is suitable for the system \eqref{eqn:e4}. For $\theta \in (0,\theta_{1}]$, we estimate
\begin{align}\label{A}
\|u(\theta)\|_\mathbb{U}&\leq \|B^{*}\|_{\mathcal{L}(\mathbb{X});\mathbb{U}}\| \mathrm{T}(\theta_{	 1}-\theta)^{*}\|_{\mathcal{L}(\mathbb{X})}\|(\Gamma^{\theta_{1}}_{0})^{-1}\|\Bigg[\|\zeta_{0}\|_{\mathbb{X}}+K\|\phi(0)\|_{\mathbb{X}}\nonumber\\&\quad+\left\|\int_{0}^{\theta_{1}} \mathrm{T}(\theta_{1}-\tau)\int_{0}^{\tau}\kappa(\tau-\eta)q(\eta,\tilde{\varkappa }_{\eta})d\eta d\tau\right\|_{\mathbb{X}}\Bigg]\nonumber\\
&\leq \frac{MK}{\delta^{0}}\left[ \|\zeta_{0}\|_{\mathbb{X}}+K\|\phi(0)\|_{\mathbb{X}} +K
	Sb\kappa_{b}\right],\nonumber\\
&\leq R_{0}.
\end{align}
For $\theta\in (\lambda_{j},\theta_{j+1}]\ \mbox{with}\ j=1,\dots,n,$ we estimate
\begin{align}\label{B}
\|u(\theta)\|_\mathbb{U}&\leq \|B^{*}\|_{\mathcal{L}(\mathbb{X});\mathbb{U}}\| \mathrm{T}(\theta_{j+	 1}-\theta)^{*}\|_{\mathcal{L}(\mathbb{X})}\|(\Gamma^{\theta_{j+1}}_{\lambda_{j}})^{-1}\|\Bigg[\|\zeta_{j}\|_{\mathbb{X}}+K\|\nu_{j}(\lambda_{j},\tilde{\varkappa }(\theta^{-}_{j}))\|_{\mathbb{X}}\nonumber\\&\quad+\left\|\int_{\lambda_{j}}^{\theta_{j+1}} \mathrm{T}(\theta_{j+1}-\tau)\int_{0}^{\tau}\kappa(\tau-\eta)q(\eta,\tilde{\varkappa }_{\eta})d\eta d\tau\right\|_{\mathbb{X}}\Bigg]\nonumber\\
&\leq \frac{MK}{\delta^{j}}\left[ \|\zeta_{j}\|_{\mathbb{X}}+KC_{\nu_{j}} +K
	Sb\kappa_{b}\right],\nonumber\\
&\leq R_{j}.
\end{align}
Combining \eqref{A} and \eqref{B}, we obtain 
\begin{align*}
\|u(\theta)\|_\mathbb{U}&\leq R,\ \mbox{for}\ \theta\in J,  \ \mbox{where} \ R=\max_{0\leq j\leq n}\{R_{j}\}.
\end{align*}
	{\bf Step 2}: Here, we show that $ \delta(\mathbb{Y})$ is bounded. Let us take $\theta\in(\lambda_{j},\theta_{j+1}],\ j=1,\dots,n$ and $y\in \mathbb{Y}$, we estimate
	 \begin{align*}
	 \|	(\delta y)(\theta)\|_{\mathbb{X}}&\leq\|\mathrm{T}(\theta-\lambda_{j})(\nu_{j}(\lambda_{j},\tilde{y}(\theta^{-}_{j})))\|_{\mathbb{X}} +\int_{\lambda_{j}}^{\theta}\| \mathrm{T}(\theta-\tau)Bu(\tau)\|_{\mathbb{X}}d\tau\\&\quad+\int_{\lambda_{j}}^{\theta}\left\| \mathrm{T}(\theta-\tau)\left[\int_{0}^{\tau}\kappa(\tau-\eta)q(\eta,\tilde{y}_{\eta})d\eta\right]\right\|_{\mathbb{X}}d\tau.\nonumber\\&\leq KC_{\nu_{j}}+ MKRb+KSb\kappa_{b}.
	 \end{align*}
	Thus, we have
	\begin{align}\label{4.25}
	\|	(\delta y)\|_{ PC}\leq\bigg (KC_{\nu_{j}}+  MKRb+KSb\kappa_{b}\bigg).
	\end{align}
	Now, for $\theta\in[0,\theta_{1}]$ and $y\in \mathbb{Y}$, we compute
	\begin{align*}
	&\|	(\delta y)(\theta)\|_{\mathbb{X}}\nonumber\\&\leq 	\|\mathrm{T}(\theta)\|\|\phi(0)\|_{\mathbb{X}}+\int_{0}^{\theta}\| \mathrm{T}(\theta-\tau)Bu(\tau)\|_{\mathbb{X}}d\tau+\int_{0}^{\theta}\left\| \mathrm{T}(\theta-\tau)\left[\int_{0}^{\tau}\kappa(\tau-\eta)q(\eta,\tilde{y}_{\eta})d\eta\right]\right\|_{\mathbb{X}}d\tau\nonumber\\&\leq K\|\phi(0)\|_{\mathbb{X}} + MKRb +KSb\kappa_{b}.
	\end{align*}
	Hence, we get
	\begin{align}\label{4.26}
	\|	(\delta y)\|_{ PC}\leq \bigg(K\|\phi(0)\|_{\mathbb{X}} +MKRb +KSb\kappa_{b}\bigg).
	\end{align}
		Similarly for $\theta\in(\theta_{j},\lambda_{j}],\ j=1,\dots,n$ and $y\in 	\mathbb{Y}$, we get
	\begin{align}\label{4.27}
	\|	( \delta y)\|_{ PC}\leq C _{\nu_{j}}.
	\end{align}
	Using the inequalities \eqref{4.25}, \eqref{4.26} and \eqref{4.27}, we obtain
	\begin{align}
	\| (\delta y)\|_{ PC}\leq \alpha_{2}\nonumber.
	\end{align}
			{\bf Step 3}: Next we'll show that the map $ \delta$ is a contraction map. Let us take any  $\varkappa,y\in 	\mathbb{Y}$, $\theta\in(\lambda_{j},\theta_{j+1}],\ j=1,\dots,n$, we compute
	\begin{align}
	&\|(\delta \varkappa)(\theta)-( \delta y)(\theta)\|_{\mathbb{X}}\nonumber\\&\le\| \mathrm{T}(\theta-\lambda_{j})(\nu_{j}(\lambda_{j},\tilde{\varkappa}(\theta^{-}_{j})))-  \mathrm{T}(\theta-\lambda_{j})(\nu_{j}(\lambda_{j},\tilde{y}(\theta^{-}_{j})))\|_{\mathbb{X}}\nonumber\\&\quad+\int_{\lambda_{j}}^{\theta}\| \mathrm{T}(\theta-\tau)\|_{\mathcal{L}(\mathbb{X})}\|B\|_{\mathcal{L}(\mathbb{U};\mathbb{X})}\|B^{*}\|_{\mathcal{L}(\mathbb{X};\mathbb{U})}\| \mathrm{T}(\theta_{j+1}-\tau)^{*}\|_{\mathcal{L}(\mathbb{X})}\|(\Gamma_{\lambda_{j}}^{\theta_{j+1}})^{-1}\|_{\mathcal{L}(\mathbb{X})}\| h_{j}(\varkappa)- h_{j}(y)\|_{\mathbb{X}}d\tau\nonumber\\&\quad+\int_{\lambda_{j}}^{\theta}\| \mathrm{T}(\theta-\tau)\|_{\mathcal{L}(\mathbb{X})}\left[\int_{0}^{\tau}|\kappa(\tau-\eta)|\|q(\eta,\tilde{\varkappa}_{\eta})-q(\eta,\tilde{y}_{\eta})\|_{\mathbb{X}}d\eta\right]d\tau\nonumber\\&\leq \| \mathrm{T}(\theta-\lambda_{j})\|_{\mathcal{L}(\mathbb{X})}\|\nu_{j}(\lambda_{j},\tilde{\varkappa}(\theta^{-}_{j}))-\nu_{j}(\lambda_{j},\tilde{y}(\theta^{-}_{j}))\|_{\mathbb{X}}\nonumber\\&\quad+\frac{M^{2}}{\delta^{j}}\int_{\lambda_{j}}^{\theta}\| \mathrm{T}(\theta-\tau)\|_{\mathcal{L}(\mathbb{X})}\| \mathrm{T}(\theta_{j+1}-\tau)^{*}\|_{\mathcal{L}(\mathbb{X})}\|h_{j}(\varkappa)- h_{j}(y)\|_{\mathbb{X}}d\tau\nonumber\\&\quad+\int_{\lambda_{j}}^{\theta}KL_{q}\kappa_{b}\|\tilde{\varkappa}_{\eta}-\tilde{y }_{\eta}\|_{D}d\tau.\label{eqn:4.20}
	\end{align}
	Using Assumptions \ref{as4.1}, we estimate
	\begin{align}\label{eqn:4.21}
	\|h_{j}(\varkappa)- h_{j}(y)\|_{\mathbb{X}}&=KL_{\nu_{j}}\|\varkappa-y\|_{\mathbb{X}}+KL_{q}b\kappa_{b}\|\tilde{\varkappa }_{\eta}-\tilde{y }_{\eta}\|_{D}.
	\end{align}
	Combine \eqref{es}, \eqref{eqn:4.20} and \eqref{eqn:4.21}, we obtain
		\begin{align}
	&\|( \delta \varkappa)(\theta)-( \delta y)(\theta)\|_{\mathbb{X}}\nonumber\\&\leq KL_{\nu_{j}}\|\varkappa-y\|_{\mathbb{X}}+ KL_{q}\kappa_{b}\gamma b\|\varkappa-y\|_{ PC} \nonumber\\&+\frac{M^{2}K^{2}b}{\delta^{j}}\bigg[KL_{\nu_{j}}\|\varkappa-y\|_{\mathbb{X}}+KL_{q}\kappa_{b}\gamma b\|\varkappa-y\|_{ PC}\bigg]\nonumber.
	\end{align}
	By above expression, we estimate
	\begin{align}\label{eqn:4.22}
&\|( \delta\varkappa)(\theta)-(\delta y)(\theta)\|_{\mathbb{X}}\nonumber\\&\leq\bigg(\bigg( KL_{\nu_{j}}+KL_{q}\kappa_{b}\gamma b\bigg) \bigg(1+\frac{M^{2}K^{2}b}{\delta^{j}}\bigg)\bigg)\big\|\varkappa-y\big\|_{ PC}.
	\end{align}
	For $\theta\in[0,\theta_{1}]$ and $\varkappa,y\in 	\mathbb{Y}$, we obtain
	\begin{align}\label{eqn:4.23}
	\nonumber &\|( \delta \varkappa)(\theta)-(\delta y)(\theta)\|_{\mathbb{X}}\\&\leq \int_{0}^{\theta}\| \mathrm{T}(\theta-\tau)\|_{\mathcal{L}(\mathbb{X})}\|B\|_{\mathcal{L}(\mathbb{U};\mathbb{X})}\|B^{*}\|_{\mathcal{L}(\mathbb{X})}\| \mathrm{T}(\theta_{1}-\tau)^{*}\|_{\mathcal{L}(\mathbb{X})}\|(\Gamma_{0}^{\theta_{1}})^{-1}\|\| h_{0}(\varkappa)- h_{0}(y)\|_{\mathbb{X}}d\tau\nonumber\\&+\int_{0}^{\theta}\| \mathrm{T}(\theta-\tau)\|_{\mathcal{L}(\mathbb{X})}\left[\int_{0}^{\tau}|\kappa(\tau-\eta)|\|q(\eta,\tilde{\varkappa}_{\eta})-q(\eta,\tilde{y}_{\eta})\|_{\mathbb{X}}d\eta\right]d\tau.\nonumber\\&\leq \frac{M^{2}}{\delta^{0}}\int_{0}^{\theta}\| \mathrm{T}(\theta-\tau)\|_{\mathcal{L}(\varkappa)}\| \mathrm{T}(\theta_{1}-\tau)^{*}\|_{\mathcal{L}(\varkappa)}\| h_{0}(\varkappa)- h_{0}(y)\|_{\mathbb{X}}d\tau\nonumber\\&+L_{q}\kappa_{b}\int_{0}^{\theta}\| \mathrm{T}(\theta-\tau)\|_{\mathcal{L}(\mathbb{X})}\|\tilde{\varkappa }_{\eta}-\tilde{y }_{\eta}\|_{D}d\tau.
	\end{align}
	Using Assumptions \ref{as4.1}, we estimate
	\begin{align}\label{eqn:4.24}
	\| h_{0}(\varkappa)- h_{0}(y)\|_{\mathbb{X}}&=KL_{q}b\kappa_{b}\|\tilde{\varkappa }_{\eta}-\tilde{y }_{\eta}\|_{D}.	
	\end{align}
		Combine \eqref{es}, \eqref{eqn:4.23} and \eqref{eqn:4.24}, we obtain
	\begin{align}\label{eqn:4.25}
	\|( \delta \varkappa)(\theta)-(\delta y)(\theta)\|_{\mathbb{X}}&\leq  \bigg(1+\frac{M^{2}K^{2}b}{\delta^{0}}\bigg)KL_{q}\kappa_{b}\gamma b\|\varkappa-y\|_{ PC}.
	\end{align}
	Similarly for $\theta\in(\theta_{j},\lambda_{j}],\ j=1,\dots,n$ and $\varkappa,y\in 	\mathbb{Y}$, we have
	\begin{align}
	\|(\delta \varkappa)(\theta)-(\delta y)(\theta)\|_{\mathbb{X}}& \leq\|\nu_{j}(\theta,\tilde{\varkappa }(\theta_{j}^{-}))-\nu_{j}(\theta,\tilde{y }(\theta_{j}^{-}))\|_{\mathbb{X}}\nonumber\\ &\leq L_{\nu_{j}}\|\tilde{\varkappa }(\theta_{j}^{-})-\tilde{y}(\theta_{j}^{-})\|_{\mathbb{X}}\nonumber\\
	&\leq L_{\nu_{j}}\|\varkappa-y\|_{\mathbb{X}}.\nonumber
	\end{align}
	Using the above estimate, we obtain
	\begin{align}\label{eqn:4.26}
	\|(\delta\varkappa)(\theta)-(\delta y)(\theta)\|_{\mathbb{X}}\leq& L_{\nu_{j}}\|\varkappa-y\|_{ PC}.
	\end{align}
	The inequalities \eqref{eqn:4.22},\eqref{eqn:4.25} and \eqref{eqn:4.26} gives
		\begin{align}\label{co}
	\|(\delta\varkappa)-(\delta y)\|_{ PC}\leq L^{'}_{F}\|\varkappa-y\|_{ PC},
	\end{align}
	where
	\begin{align}
	L^{'}_{F}&=\max\bigg\{\max_{1\leq j \leq n}\bigg(\bigg( KL_{\nu_{j}}+KL_{q}\kappa_{b}\gamma b\bigg) \bigg(1+\frac{M^{2}K^{2}b}{\delta^{j}}\bigg)\bigg), \nonumber\\&\qquad\qquad\bigg(1+\frac{M^{2}K^{2}b}{\delta^{0}}\bigg)KL_{q}\kappa_{b}\gamma b,\ \max_{1\leq j \leq n}L_{\nu_{j}}\bigg\}.\nonumber
	\end{align}
	By condition \eqref{LF1}, we have $L^{'}_{F}<1$. Hence, by the estimate \eqref{co}, we conclude that $\delta$ is a contraction map. Then $\delta$ has a unique fixed point. 
Which is a mild solution of (\ref{eqn:e4}). Thus, system (\ref{eqn:e4}) is totally controllable.	
	\end{proof}
\section{Application}\label{sec5}
\begin{Ex}Let us consider the following transport partial differential equation with impulsive effect and nonlocal condition:
\begin{equation}\label{exp}
\left\{
\begin{aligned}	
\frac{\partial }{\partial \theta}z(\theta,\hslash)&=\frac{\partial}{\partial\hslash}z(\theta, \hslash)+\mu(\theta, \hslash)+F(\theta, z(\theta-\beta, \hslash)),\;
0\le \hslash\le\pi,\\ &\qquad \qquad \theta\in (\lambda_{j},\theta_{j+1}),\ j=0,\dots,n,\ \beta>0, \\
	z(0,\hslash)&=0, \  0\le \hslash\le\pi, \\
z(\theta,\hslash)&=\nu_{j}(\theta,z(\theta^{-}_{j},\hslash)),\;\theta\in (\lambda_j,\theta_j],\ j=1,\dots,n,\ 0\le \hslash\le\pi, \\
z(\theta,\hslash)&=\phi(\theta,\hslash)+\nu(z(\theta,\hslash)),\; \theta \in [-\beta, 0], \ 0\le \hslash\le\pi,
\end{aligned}
\right.
\end{equation}
where $\phi: [-\beta,0]\times [0, \pi]\rightarrow \mathbb{R}$, is piecewise continuous function.
\end{Ex}
 Take $J=[0,b]$, $\mathbb{X}=\mathbb{U}= L^{2}([0,\pi];\mathbb{R}).$
Let the operator $ A : D(A)\subset \mathbb{X} \rightarrow \mathbb{X}$ be defined as
\begin{align*}
Av = \frac{\partial }{\partial\hslash}v, \  \text{ where }  D(A) := \{v\in H^{1}((0, \pi);\mathbb{R}) : v(\pi)=0\}.
\end{align*}
Note that, $C_0^\infty([0,\pi];\mathbb{R})\subset D(A)$ and hence $D(A)$ is dense in $\mathbb{X}$. It can be easily verified that the operator $A$ with this domain is closed, similarly as proved in Example A.3.47 \cite{curtain2012introduction}. Adjoint of operator $A$ is given by
\begin{align*}
A^{*}v = -\frac{\partial }{\partial\hslash}v, \  \text{ where }
D(A^{*}) := \{v\in H^{1}((0, \pi);\mathbb{R}) : v(0)=0\}.
\end{align*}
Moreover,
\begin{align*}
Re(<Av,v>)=-\frac{1}{2}|z(0)|^{2}\le 0,
\end{align*}
and
\begin{align*}
Re(<A^{*}v,v>)=-\frac{1}{2}|z(\pi)|^{2}\le 0.
\end{align*}
Hence by applying corollary 2.2.3 \cite{curtain2012introduction}, we obtain that $A$ is the infinitesimal generator of a contraction semigroup $ \mathrm{T}(\theta)$ on $\mathbb{X}$.
The semigroup $( \mathrm{T}(\theta))_{t\geq0}$ (see, Example 2.3.8 in \cite {Tucsnak}) is given by
	\begin{equation}
( \mathrm{T}(\theta)v)(\hslash)=
\left\{
\begin{aligned}
\begin{array}{ll}
v(\hslash+\theta),&\mbox{ if } \;(\hslash+\theta) \leq\pi,\\
0,& \mbox{ if }\;(\hslash+\theta)  \textgreater \pi,
\end{array}
\end{aligned}
\right.
\end{equation}
where $\theta\in J.$ The semigroup $\{ \mathrm{T}(\theta\}_{\theta\geq0}$ is not compact on $\mathbb{X}$ (see Example 1.27 in \cite{NR}).
Next, let us define the  following operator $ B: L^{2}([0, \pi];\mathbb{R})\rightarrow \mathbb{X}$, such that
\begin{align}
B(u(\theta))(\hslash)=u(\theta)(\hslash)={\mu}(\theta,\hslash),\ \theta\in J,\ \hslash\in [0,\pi].
\end{align}
Let the functions   $ \varkappa: J \rightarrow \mathbb{X} $ and $\phi:[-\beta,0]\rightarrow  \mathbb{X}$ be given by
\begin{align*}
{\varkappa}(\theta)(\hslash)&=z(\theta,\hslash),\ \hslash\in [0,\pi],\\
{\phi}(\theta)(\hslash)&=\phi(\theta, \hslash),\  \hslash\in [0,\pi].
\end{align*}
We now consider non-instantaneous impulse $$\nu_{j}(\theta,\varkappa(\theta_i^-))(\hslash):=\nu_{j}(\theta,z(\theta^{-}_{j},\hslash)),\ j = 1,\dots, n.$$ Let us choose
\begin{align*}
\nu_{j}(\theta,\varkappa)= \theta \varkappa,\ \mbox{for}\ \theta\in (\theta_j,\lambda_j],\ j=1,\dots,n.
\end{align*}
We now calculate
\begin{align*}
\|\nu_{j}(\theta,\varkappa)-\nu_{j}(\theta,y)\|_{\mathbb{X}}&=\|\theta \varkappa-\theta y\|_{\mathbb{X}}=|\theta|\|\varkappa-y\|_{\mathbb{X}}\\
&\le b\|\varkappa-y\|_{\mathbb{X}}.
\end{align*}
Hence, impulse functions satisfy (H3).
\begin{itemize}
\item[\textbf{Case 1:}]
First, let us consider the following nonlinear function and and nonlocal condition given by

The nonlinear function $\eta: J_{1}=\bigcup^{n}_{j=0}[\lambda_{j},\theta_{j+1}] \times D\rightarrow  \mathbb{X}$ is defined as
$$F(\theta, z(\theta-\beta, \hslash)) =k_{0}\sin(z(\theta-\beta),\hslash),$$
$$\eta(\theta,\varkappa_{\theta})(\hslash)=F(\theta, z(\theta-\beta, \hslash))=k_{0}\sin(\varkappa_{\theta}), \ \hslash\in[0,\pi],$$
where $D$ is defined in \eqref{12}.
We now estimate
\begin{align*}
\left\|\eta(\theta,\varkappa_{\theta})-\eta(\theta,y_{\theta})\right\|_{\mathbb{X}}&=\left\| k_{0}\sin(\varkappa_{\theta})-k_{0}\sin(y_{\theta})\right\|_{\mathbb{X}}\\
&= \left\|k_{0}2\cos\left(\frac{\varkappa_{\theta}+y_{\theta}}{2}\right) \sin\left(\frac{\varkappa_{\theta}-y_{\theta}}{2}\right)\right\|_{\mathbb{X}}\\
&\le{k_{0}}\left\|\varkappa_{\theta}-y_{\theta} \right\|_{D}.
\end{align*}
Also, {$\|\eta(\theta,\varkappa_{\theta}\|\leq k_0.$} Hence, the nonlinear function satisfies the Assumption (H2).

Now we consider $ \nu:PC([0,\tau];\mathbb{X})\rightarrow \mathbb{X}$ such that
\begin{align*}
(\nu(\varkappa))(\theta)(\hslash):=\nu(z(\theta, \hslash)), \ \theta \in [0, \tau].
\end{align*}
We can choose the function $\nu$ as
\begin{align*}
(\nu(\varkappa))(\theta)(\hslash):= \sum_{j=1}^{n}\alpha_{j}z(\theta_{j}, \hslash),\ \theta_{j}\in J,\ \hslash\in [0,\pi],
\end{align*}
where $\alpha_{j}'s$ are small enough, then $\xi$ satisfy Hypothesis (H5) (see \cite{LL}).//
Substituting all of the above expressions in the system (\ref{exp}), it can be expressed as an abstract form given in (\ref{eqn:e1}) satisfying the Assumptions \ref{as2.1}.  Finally, by Theorem \ref{jt3}, the semilinear system \eqref{exp} (equivalent to the system \eqref{eqn:e1}) is totally controllable.
\item[\textbf{Case 2:}] Secondly, let us define nonlinear function and nonlocal condition as follows:
\begin{align}\label{nl}
F(\theta, z(\theta-\beta, \hslash))=\int_{0}^{\theta}(\theta-s)\rho(s,z(s-\beta,\hslash))ds,
\end{align}
$$\text{where }\rho(s,z(s-\beta,\hslash))=\dfrac{e^{-\theta}\vert z(\theta-\beta,\xi)\vert}{(a+2e^{\theta})(1+2\vert z(\theta-\beta,\xi)\vert)},\ a>-1.$$
Let us define $q: J_{1}=\bigcup^{n}_{j=0}[\lambda_{j},\theta_{j+1}] \times D\rightarrow  \mathbb{X}$, $$ q(\theta,\varkappa_{\theta})(\hslash)=\rho(s,z(s-\beta,\hslash))=\dfrac{e^{-\theta}\vert x_{\theta}\vert}{(a+2e^{\theta})(1+2\vert \varkappa_{\theta}\vert)}$$
\begin{align*}
\left\|q(\theta,\varkappa_{\theta})-q(\theta,y_{\theta})\right\| &=\left\|\dfrac{e^{-\theta}\vert \varkappa_{\theta}\vert}{(a+2e^{\theta})(1+2\vert \varkappa_{\theta}\vert)}-\dfrac{e^{-\theta}\vert y_{\theta}\vert}{(a+2e^{\theta})(1+2\vert y_{\theta}\vert)}\right\|\\
&\leq\left\vert\dfrac{e^{-\theta}}{(a+2e^{\theta})}\right\vert\left\|\dfrac{\vert \varkappa_{\theta}\vert}{(1+2\vert \varkappa_{\theta}\vert)}-\dfrac{\vert y_{\theta}\vert}{(1+2\vert y_{\theta}\vert)}\right\|\\
&\leq L_{q}\left\|\varkappa_{\theta}-y_{\theta}\right\|,
\end{align*}
for $\varkappa,y\in \mathbb{X}$ and $L_{q}=\dfrac{1}{a+2}$ and $\|q(\theta,\varkappa_{\theta})\|\leq 1.$ Hence Assumption (A2) is satisfied.

Now, the nonlocal condition is defined as
\begin{align}\label{nc}
\xi(\varkappa)(\theta)(\hslash):=0.
\end{align}
Partial differential system (\ref{exp}) can be expressed in abstract form (\ref{eqn:e4}), by substituting nonlinear function \eqref{nl}, nonlocal condition \eqref{nc} and all other substitution are same as in case 1. Finally, by applying the Theorem \ref{jt6}, we can conclude that the semilinear system \eqref{exp} (equivalent to \eqref{eqn:e4}) is totally  controllable.
\end{itemize}
\medskip\noindent
{\bf Acknowledgments:} Corresponding author would like to thank the Council of Scientific and Industrial Research, New Delhi, Government of India, project (ref. no. 25(0315)/20/EMR-II).
The authors wishes to express special thanks to Department of Mathematics, Indian Institute of Technology Roorkee-IIT Roorkee, for providing stimulating scientific environment and resources.


\begin{thebibliography}{10}
	
	\bibitem{HMA}
	Ahmed, H.M., El-Borai, M.M., El Bab, A.O. and Ramadan, M.E., Approximate controllability of noninstantaneous impulsive Hilfer fractional integrodifferential equations with fractional Brownian motion, \emph{Boundary Value Problems}, 2020, 2020(1), 1-25.
	\bibitem{AS}
	\text{Arora, S., Singh, S., Dabas, J. and Mohan, M.T.}, Approximate controllability of semilinear impulsive functional differential systems with nonlocal conditions, \emph{IMA Journal of Mathematical Control and Information}, 2020.
		\bibitem{azbelev2007introduction}
	Azbelev, N.V. and Rakhmatullina, L., \emph{Introduction to the Theory of Functional Differential Equations: Methods and Applications}, Hindawi Publishing Corporation, 2007.
	\bibitem{balachandran2000controllability}
	Balachandran, K., Sakthivel, R. and Dauer, J.P., Controllability of neutral
		functional integrodifferential systems in Banach spaces, \emph{Computers \&
		Mathematics with Applications}, 2000, 39(1-2), 117-126.
	
	
	\bibitem{benchohra2004controllability}
	Benchohra, M., Gorniewicz, L., Ntouyas, S., K. and Ouahab, A., Controllability
	results for impulsive functional differential inclusions, \emph{Reports on
		Mathematical Physics}, 2004, 54(2), 211-228.
		\bibitem{benchohra2000existence}
	Benchohra, M., and Ntouyas, S.K., Existence of solutions of nonlinear
	differential equations with nonlocal conditions, \emph{Journal of Mathematical
		Analysis and Applications}, 2000, 252(1), 477-483.
	
	
		\bibitem{Byszewski91}
		Byszewski, L., Theorem about the existence and uniqueness of solution of a semilinear nonlocal cauchy problem nonlinear abstract differential equation with deviated argument,  \emph{Journal of Mathematical Analysis and Applications}, 1991, 162(2), 494-505.
			\bibitem{RK} 
		Chalishajar, D.N., George, R.K., Nandakumaran, A.K. and Acharya, F.S., Trajectory Controllability of nonlinear integro-differential system, \emph{Journal of the Franklin Institute}, 2010, 347(7), 1065-1075.
			\bibitem{D Chalishjar}
		Chalishajar, D.N. and Kumar, A., Total controllability of the second order semi-linear differential equation with infinite delay and non-instantaneous impulses, \emph {Mathematical and Computational Applications}, 2018, 23(3), 32.
		
	\bibitem{chang2007controllability}
	Chang, Y. K., Controllability of impulsive functional differential systems
	with infinite delay in Banach spaces, \emph{Chaos, Solitons \& Fractals},
	2007, 33(5), 1601-1609.
	
		\bibitem{curtain2012introduction}
	Curtain, R.F. and Zwart, H., \emph{An Introduction to Infinite-Dimensional
		Linear Systems Theory}, Springer Science \& Business Media, 2012.
	
	\bibitem{FM}	 Feckan, M. and Wang, J., A general class of impulsive evolution equations, \emph{Topological Methods in Nonlinear Analysis}, 2015, 46(2), 915-933.
	\bibitem{FX}
			\text{Fu, X. and  Liu, X.}, Existence of periodic solutions for abstract neutral non-autonomous equations with infinite delay, {\em Journal of Mathematical Analysis and Applications}, 2007, 325 (1), 249–267.
	\bibitem{hale2013introduction}
	Hale, J.K., Verduyn, L.S. and Lunel, S.M.V., \emph{Introduction to Functional Differential Equations},  Springer Science \& Business Media, 1993.
	
	\bibitem{EH}	 Hernández, E. and O’Regan, D., On a new class of abstract impulsive differential equations, \emph{Proceedings of the American Mathematical Society}, 2013, 141(5), 1641-1649.
	\bibitem{ji2011controllability}
	Ji, S., Li, G., and Wang, M., Controllability of impulsive differential systems
	with nonlocal conditions,  \emph{Applied Mathematics and Computation},
	2011, 217(16), 6981-6989.
	
	
	\bibitem{JK}
	Klamka, J., Constrained controllability of semilinear systems with delays, \emph{Nonlinear Dynamics}, 2009, 56(1-2), 169-177.
		\bibitem{KA}
	Kumar, A., Malik, M. and Sakthivel, R.,  Controllability of the second-order nonlinear differential equations with non-instantaneous impulses, \emph{Journal of Dynamical and Control Systems}, 2018, 24(2), 325-342.
	
		\bibitem{RD}
	Kumar, P., Haloi, R., Bahuguna, D. and Pandey, D.N., Existence of solutions to a new class of abstract non-instantaneous impulsive fractional integro-differential equations, \emph{Nonlinear Dynamics and Systems Theory}, 2016, 16(1), 73-85.
	
	\bibitem{MV}
Kumar, V. and Malik, M., Total controllability and observability for dynamic systems with non-instantaneous impulses on time scales, \emph {Asian journal of control}, 2019, 23(2), 847-869.
\bibitem{VM}
Kumar, V. and Malik, M., Total controllability results for a class of time-varying switched dynamical systems with impulses on time scales, \emph {Asian journal of control}, 2020, 1-9.



\bibitem{LL}
 Liang, J., Liu, J.H., and  Xiao, T.J., Nonlocal impulsive problems for nonlinear differential equations in Banach spaces, \emph {Mathematical and Computer Modelling}, 2009, 49 (3-4), 798–804.
\bibitem{LA}
		 \text{Lunardi, A.,} On the linear heat equation with fading memory, {\em SIAM Journal on Mathematical Analysis}, 1990, 21 (5), 1213–1224.
\bibitem{MM}
	Malik, M., Dhayal, R. and Abbas, S., Exact controllability of a retarded fractional differential equation with non-instantaneous impulses, \emph{Dynamics of Continuous Discrete Impulsive Systems Series B: Applications \& Algorithms}, 2019, 26(1), 53-69.
	 \bibitem{Muslim2018}
	\text{Malik, M., Kumar, A. and Feckan, M.}, Existence, uniqueness, and stability of solutions to second-order nonlinear differential equations with non-instantaneous impulses,  \emph{Journal of King Saud University-Science}, 2018, {30}, 204–213.
	
	
	
	
\bibitem{NR} \text{Nagel, R.}, \emph{One-Parameter Semigroups of Positive Operators}, Lecture Notes In Mathematics 1184, Springer Verlag, Berlin-Heidelberg- New York-Tokyo, 1986.			
%
	
		\bibitem{ntouyas1997global}
	Ntouyas, S.K. and  Tsamatos, P.C., Global existence for semilinear evolution
	equations with nonlocal conditions,  \emph{Journal of Mathematical Analysis and
		Applications}, 1997, 210(2), 679--687.
	\bibitem{JW}
					\text{ Nunziato, J.W.,}  On heat conduction in materials with memory, {\em Quarterly of Applied Mathematics }, 1971, 29 (2),  187–204.
	\bibitem{PJY}
Park, J.Y., Balachandran, K. and Arthi, G., Controllability of impulsive neutral integro-differential systems with infinite delay in Banach spaces, \emph{Nonlinear Analysis: Hybrid Systems}, 2009, 3(3), 184-194.

	\bibitem{AP}
	Pazy, A., \emph{Semigroups of Linear Operators and Applications to Partial Differential Equations}, Springer Science \& Business Media, 2012.
	
	\bibitem{PM}	Pierri, M., Henríquez, H.R. and Prokopczyk, A., Global solutions for abstract differential equations with non-instantaneous impulses, \emph{Mediterranean Journal of Mathematics}, 2016, 13(4), 1685-1708.
	
	\bibitem{Pierri2013}
	Pierri, M., Regan, D.O. and Rolnik, V., Existence of solutions for semi-linear abstract differential equations with non-instantaneous impulses, \emph{Applied Mathematics and  Computations}, 2013, 219(12), 6743-6749.
	
	\bibitem{SLS}
	Shen, L., Shi, J. and Sun, J., Complete controllability of impulsive stochastic integro-differential systems, \emph{Automatica}, 2010, 46(6), 1068-1073.
	\bibitem{Tucsnak}
     Tucsnak, M., Weiss, G., \emph {Observation and Control for Operator Semigroups},
	Birkh\''auser Verlag AG, Basel, 2009.
	
\bibitem{WJ}
\text{Wang, J., Ibrahim, A.G., Fe\v{c}kan, M., Zhou, Y.}, Controllability of fractional non-instantaneous impulsive
differential inclusions without compactness, \emph{IMA Journal of Mathematical Control and Information}, 2017, 36(2), 443-460.
	
	\bibitem{Zhang2017}
	Zhang, X., Li, Y. and Chen, P., Existence of extremal mild solution for the initial value problem of evolution equation with non-instantaneous impulses, \emph{Journal of Fixed Point Theory and Applications}, 2017, 19, 3013-3027.

	\bibitem{ZH}
	Zhu, B., Han, B., Liu, L. and Yu, W., On the fractional partial integro-differential equations of mixed type with non-instantaneous impulses, \emph{Boundary Value Problems}, 2020, 154. 
	
\end{thebibliography}
\end{document}